\documentclass{amsart} 

\usepackage[usenames, dvipsnames]{color}
\usepackage{amsfonts}
\usepackage{amsthm}
\usepackage{amsmath}
\usepackage{amsfonts}
\usepackage{latexsym}
\usepackage{amssymb}
\usepackage{amscd}
\usepackage[latin1]{inputenc}
\usepackage{verbatim}
\usepackage{enumerate}
\usepackage{comment}
\usepackage{mathrsfs}

\newtheorem{theorem}{Theorem}[section]

\newtheorem{proposition}[theorem]{Proposition}
\newtheorem{corollary}[theorem]{Corollary}
\newtheorem{example}[theorem]{Example}
\newtheorem{definition}[theorem]{Definition}

\theoremstyle{definition}

\newcommand\tr{ \mbox{Tr} }
\newcommand\leac{\ll_{\rm ac}}

\renewcommand{\H}{\mathcal{H}}

\newcommand{\essran}{\operatorname{ess-ran}}
\newcommand{\esssup}{\operatorname{ess-sup}}
\newcommand{\esssupp}{\operatorname{ess-supp}}

\begin{document}

\title{An operator-valued Lyapunov theorem}

\author[Sarah Plosker]{Sarah Plosker\textsuperscript{1}}

\author[Christopher Ramsey]{Christopher Ramsey\textsuperscript{1,2}}

\thanks{\textsuperscript{1}Department of Mathematics and Computer Science, Brandon University,
Brandon, MB R7A 6A9, Canada}
\thanks{\textsuperscript{2}Department of Mathematics and Statistics, MacEwan University, Edmonton, AB  T5J 4S2, Canada}


\maketitle

\begin{abstract}   
 We generalize Lyapunov's convexity theorem for classical (scalar-valued)   measures to quantum (operator-valued)   measures. In particular, we show that the range of a nonatomic quantum probability measure is a weak$^*$-closed convex set of quantum effects (positive operators bounded above by the identity operator) under a sufficient condition on the non-injectivity of integration. To prove the operator-valued version of Lyapunov's theorem, we must first define the notions of essentially bounded, essential support, and essential range for quantum random variables (Borel measurable functions from a set to the bounded linear operators acting on a Hilbert space). 
\end{abstract}


\section{Introduction} \label{sec:intro}

Lyapunov's theorem (\cite{L1}, and Corollary \ref{thm:ClassLyapunov} below) states that if $\mu_1$, $\mu_2$, $\dots$,  $\mu_n$ are finite positive nonatomic measures on a measurable space $(X,\Sigma)$, then the set of points in $\mathbb{R}^n$ of the form $(\mu_1(A), \mu_2(A), \dots, \mu_n(A))$, with $A$ ranging over the measurable subsets of $X$, is closed and convex. In particular, the range of a single finite positive nonatomic measure is closed and convex. Our goal is to extend this result into the quantum, or operator-valued, setting. 

In quantum mechanics,  a quantum system is described by a   Hilbert space $\mathcal H$. Positive operator-valued measures (POVMs)  and quantum probability measures (POVMs that sum to the identity) are studied in quantum information theory in order to fully describe measurements of a quantum system \cite{busch1996quantum,davies1976quantum,holevo2003statistical}. The more general notion of operator-valued measures (OVMs) has been studied independently of quantum information theory; see  \cite{Larsonetal, Roth, VOVMbook}.

Denote the  C$^*$-algebra  of all bounded operators acting on $\mathcal H$   by $\mathcal B(\mathcal H)$ and recall that the predual of  $\mathcal B(\mathcal H)$ is  $\mathcal{B}(\mathcal H)_*=\mathcal T(\mathcal H)$, the ideal of trace class operators. In the case when $\mathcal H$ is finite-dimensional,  $\mathcal B(\mathcal H)$  and $\mathcal T(\mathcal H)$  coincide as sets, and so often in the finite-dimensional setting authors do not distinguish between the two. Unless stated otherwise, we take $\mathcal H$ to be infinite-dimensional herein. Denote the set of all positive operators of unit trace by $S(\mathcal H)\subset \mathcal T(\mathcal H)$; this is a convex subset of  $\mathcal T(\mathcal H)$. General elements of $S(\mathcal H)$, often denoted by $\rho$, are called \emph{states} or \emph{density operators} and play a critical role in quantum information theory.  Finally, let $X$ be a Hausdorff space and $\mathcal O(X)$  a $\sigma$-algebra of subsets of $X$ (often  $\mathcal O(X)$ is taken to be the  $\sigma$-algebra of Borel sets of $X$ as it will be in the next section).

\begin{definition}\cite{Larsonetal, Paulsen, MPR}
A map $\nu : \mathcal{O}(X) \to \mathcal{B}(\mathcal{H})$ is an \emph{operator-valued measure (OVM)} if it is weakly countably additive: for every countable collection $\{E_k\}_{k \in \mathbb N} \subseteq \mathcal{O}(X)$ with $E_j \cap E_k = \emptyset$ for $j \neq k$ we have
\[
\nu\left(\bigcup_{k\in \mathbb N} E_k \right) = \sum_{k \in \mathbb N} \nu(E_k)\,,
\]
where the convergence on the right side of the equation above is with respect to the ultraweak topology of $\mathcal{B}(\mathcal{H})$. An OVM  $\nu$ is 
\begin{enumerate}[(i)]
    \item \emph{bounded} if $\sup\{\|\nu(E)\| : E\in \mathcal O(X)\} < \infty$,
    \item \emph{self-adjoint} if $\nu(E)^* = \nu(E)$, for all $E\in\mathcal O(X)$,
    \item \emph{positive} if $\nu(E) \in \mathcal{B}(\mathcal H)_+$, for all $E\in \mathcal O(X)$,
    \item \emph{spectral} if $\nu(E_1 \cap E_2) = \nu(E_1)\nu(E_2)$, for all $E_1,E_2\in \mathcal O(X)$
\end{enumerate}
Moreover, an OVM $\nu$ is called a \emph{positive operator-valued measure} (POVM) if it is positive (note that this necessarily implies $\nu$ is bounded), a \emph{positive operator-valued probability measure}  or \emph{quantum probability measure} if it is positive and $\nu(X) = I_{\mathcal{H}}$ (the identity operator in $\mathcal{B}(\mathcal H))$, and is called a \emph{projection-valued measure (PVM)} if it is self-adjoint and spectral (PVMs show up as an important subclass of POVMs in quantum information theory).
\end{definition}

We briefly recall some useful definitions from classical measure theory. 
Let  $X$ be a set, $\Sigma$ a $\sigma$-algebra over $X$,  and $\mu$ a measure on the measurable space $(X,\Sigma)$.  An \emph{atom} for $\mu$ is a set $E$ of nonzero measure such that, for any subset $F\subset E$, either $\mu(F)=0$ or $\mu(F)=\mu(E)$.  A measure   $\mu$ is \emph{atomic} if every set of nonzero measure contains an atom. A measure $\mu$ is \emph{nonatomic} if it has no atoms; that is, if every set of nonzero measure has a subset of different nonzero measure. Note that a measure that is not atomic is not necessarily nonatomic, and vice-versa. A classical partial order on measures is that of absolute continuity: if $\mu_1, \mu_2$ are measures on $(X,\Sigma)$, then 
$\mu_1$ is \emph{absolutely continuous with respect to} $\mu_2$,
denoted by $\mu_1\leac\mu_2$, if $\mu_1(E)=0$ for all $E\in \Sigma$ for which $\mu_2(E)=0$.

The definitions of atomic, nonatomic, and absolutely continuous can be readily adapted, mutatis mutandis, to OVMs; in fact, absolute continuity can be used to compare OVMs from the same space $X$ into different Hilbert spaces, in particular one can compare an OVM with a (classical) measure.

Given an OVM  $\nu: \mathcal O(X) \rightarrow \mathcal{B}(\mathcal H)$, for every state $\rho\in S(\mathcal H)$, we can associate to $\nu$ the induced complex measure $\nu_\rho$ on $X$ defined by
\[
\nu_\rho(E) = \tr(\rho \nu(E)) \,\, \forall E\in \mathcal O(X).
\]
If $\nu$ is a quantum probability measure, then the complex measure $\nu_\rho$ can be interpreted as mapping the measurement event $E$ to the corresponding measurement statistics;  the probability that event $E$ is measured by the quantum probability measure $\nu$ when the system $\mathcal H$ is in state $\rho$ is  $\tr(\rho\nu(E))$.   
Note that $\nu_\rho$ is countably additive since   $\nu$ is weakly countably additive.
It is not difficult to see that $\nu$ and $\nu_\rho$ are  mutually absolutely continuous for any full-rank $\rho\in S(\mathcal H)$.

Assume $\mathcal H$ is separable and let $\{e_n\}$ be an orthonormal basis. As in \cite{FPS, MPR}, denote $\nu_{ij}$ the complex measure   $\nu_{ij}(E) = \langle \nu(E)e_j,e_i\rangle, E\in \mathcal O(X)$. Now, for any full-rank density operator $\rho$ we have $\nu_{ij} \ll_{\rm ac} \nu_\rho$. Thus, by the classical Radon-Nikod\'ym theorem, there is a unique $\frac{d\nu_{ij}}{d\nu_\rho} \in L_1(X, \nu_\rho)$ such that
\[
\nu_{ij}(E) = \int_E \frac{d\nu_{ij}}{d\nu_\rho} d\nu_\rho, \ E\in \mathcal O(X).
\]

\begin{definition}\cite[Definition 2.3]{MPR} Let $\nu:\mathcal{O}(X)\rightarrow \mathcal B(\mathcal H)$ be a POVM and $\rho\in S(\mathcal{H})$ be a full-rank density operator. 
The {\em Radon-Nikod\'ym derivative} of $\nu$ with respect to $\nu_\rho$ is defined to be
\[
\frac{d\nu}{d\nu_\rho} = \sum_{i,j\geq 1} \frac{d\nu_{ij}}{d\nu_\rho} \otimes e_{ij},
\]
where $\{e_{ij}\}$ are the matrix units.
If  $\frac{d\nu}{d\nu_\rho}$  is a valid quantum random variable (in the sense that  it takes every $x$ to a bounded operator), then  $\frac{d\nu}{d\nu_\rho}$   is said to exist.
\end{definition}

If $\nu$ is into a finite-dimensional Hilbert space then $\frac{d\nu}{d\nu_\rho}$ always exists (see \cite{FPS} for further analysis into the finite-dimensional setting). However, for infinite dimensions this derivative is potentially into the unbounded operators and does not exist in such cases. 
If $\frac{d\nu}{d\nu_{\rho_0}}$ exists for one full-rank density operator $\rho_0\in S(\mathcal H)$, then it exists for all full-rank  $\rho\in S(\mathcal H)$ \cite[Corollary 2.13]{MPR}.

The following is a truly infinite-dimensional example, in the sense that it is not merely finite-dimensional living in infinite dimensions, where the 
Radon-Nikod\'ym derivative
exists. 

\begin{example}
Let $\mu$ be Lebesgue measure on $[0,1]$ and let $\{\mu_n\}_{n\geq 1}$ be restrictions of $\mu$ to the intervals $[1/(n+1), 1/n]$, respectively. Define $\nu(E) = {\rm diag}\left(\mu(E), \mu_1(E), \mu_2(E), \dots\right)$ and so $\nu : \mathcal O(X) \rightarrow \mathcal B(\mathcal H)$ is a nonatomic POVM. Taking  $\rho = {\rm diag}(1/2,1/4,\dots)$, we can see that 
\begin{eqnarray*}
\nu_\rho&=&\frac12\mu+\sum_n\frac1{2^{n+1}}\mu_n\\
&=&  \sum_n \frac{2^n+1}{2^{n+1}} \mu_n \textnormal{ and}\\
\frac{d\nu}{d\nu_\rho}&=&\sum_n \frac{2^{n+1}}{2^n+1}\chi_{[1/(n+1), 1/n]} e_{nn},
\end{eqnarray*}
where $e_{nn}$ is the $(n,n)$-matrix unit. Hence, the Radon-Nikod\'ym derivative exists.
\end{example}


An important condition that arises in the study of Lyapunov's theorem (both the operator-valued version herein and the classical version) is a condition on the non-injectivity of integration. This will be described in more detail at the end of Section 2.

Lyapunov himself found a counterexample to the convexity theorem when there is an infinite number of classical measures \cite{L2}. \cite[Chapter IX]{DiestelUhl} also contains counterexamples. These examples all fail because of injectivity. 

In Section \ref{sec:intess} we define quantum random variables, the integral of such a function against a POVM, and concepts in relation to such functions such as  essential support,  essential range, and essentially bounded. In Section \ref{sec:Lyapunov} we present our main results   extending Lyapunov's theorem to positive operator-valued measures. In particular, Corollary \ref{thm:quantumLyapunov}   shows that the range of a nonatomic quantum probability measure is a weak$^*$-closed convex set, which can be viewed as the quantum analogue to Lyapunov's theorem.

\section{Integration and essentially bounded functions}\label{sec:intess}

\begin{definition}
A  Borel measurable function  $f: X \rightarrow \mathcal{B}(\mathcal H)$  between the $\sigma$-algebra generated by the open sets of $X$ and the $\sigma$-algebra generated by the open sets of $\mathcal{B}(\mathcal H)$ is a  {\em quantum random variable}.

Equivalently, $f$ is a quantum random variable if and only if 
\[
x\mapsto \tr(\rho f(x))
\]
are Borel measurable functions for every state $\rho \in S(\mathcal H)$.

A quantum random variable $f: X \rightarrow \mathcal{B}(\mathcal H)$ is 
\begin{enumerate}[(i)]
    \item \emph{bounded} if $\sup\{\|f(x)\| : x\in X\} < \infty$,
    \item \emph{self-adjoint} if $f(x) = f(x)^*$, for all $x\in X$,
    \item \emph{positive} if $f(x) \in \mathcal{B}(\mathcal H)_+$, for all $x\in X$.
\end{enumerate}

\end{definition}
Observe that the complex-valued functions $x\mapsto \tr(\rho f(x))$ are measurable and are thus (classical) random variables. See \cite{FPS, FK, FKP, JK} for more detailed analyses of quantum random variables.

We are interested in integrating a quantum random variable  with respect to a positive operator-valued measure:
\begin{definition}\cite{FPS, MPR} 
Let $\nu : \mathcal O(X) \rightarrow \mathcal B(\mathcal H)$ be a POVM such that $\frac{d\nu}{d\nu_\rho}$ exists.
A positive quantum random variable $f: X \rightarrow \mathcal{B}(\mathcal H)$ is   {\em $\nu$-integrable} if the function
\[
f_s(x) = \tr\left( s\left(\frac{d\nu}{d\nu_\rho}(x) \right)^{1/2} f(x) \left(\frac{d\nu}{d\nu_\rho}(x) \right)^{1/2} \right), \ x\in X
\]
is $\nu_\rho$-integrable for every state $s\in S(\mathcal H)$. 
\end{definition}

If $f: X \rightarrow \mathcal{B}(\mathcal H)$ is a self-adjoint quantum random variable, then $f_+, f_- : X \rightarrow \mathcal{B}(\mathcal H)_+$ defined by 
\[
f_+(x) = f(x)_+ \ \ \textrm{and} \ \ f_-(x) = f(x)_-, \ \ x\in X
\]
are positive quantum random variables \cite[Lemma 2.1]{MPR}. An arbitrary quantum random variable $f: X \rightarrow \mathcal{B}(\mathcal H)$ is then said to be $\nu$-integrable if and only if $({\rm Re} f)_+, ({\rm Re} f)_-, ({\rm Im} f)_+$ and $({\rm Im} f)_-$ are $\nu$-integrable.

\begin{definition}\cite{FPS, MPR}
Let $\nu : \mathcal O(X) \rightarrow \mathcal B(\mathcal H)$ be a POVM such that $\frac{d\nu}{d\nu_\rho}$ exists. If $f: X \rightarrow \mathcal B(\mathcal H)$ is a $\nu$-integrable quantum random variable then the integral of $f$ with respect to $\nu$, denoted $\int_X f d\nu$, is implicitly defined by the formula
\[
\tr\left( s\int_X f d\nu\right) = \int_X f_s d\nu_\rho.
\]
\end{definition}

This definition of the integral of a quantum random variable with respect to a POVM was first defined in the finite-dimensional case in \cite{FPS}. The infinite case was shown to work in \cite{MPR} and one should note that showing that the above definition works takes some proving. 

A good sign that this is a reasonable definition is that for every set $E\in \mathcal O(X)$ we have that
\[
\int_X \chi_E(x)I_\mathcal H d\nu(x) = \nu(E).
\]
For a detailed discussion of this integral and a proof of the above consult \cite{MPR}.

\begin{definition}
Let  $\nu: \mathcal O(X) \rightarrow \mathcal{B}(\mathcal H)$ be a POVM and let    $f: X \rightarrow \mathcal{B}(\mathcal H)$ be a quantum random variable. Then the \emph{essential support} of $f$ is 
 \[
 \esssupp f=X\setminus \bigcup_{\substack{E\in \mathcal O(X) \textnormal{ open}\\ f=0 \textnormal{ a.e.\ in } E}}E.
 \]
\end{definition}

\begin{definition}
Let  $\nu: \mathcal O(X) \rightarrow \mathcal{B}(\mathcal H)$ be an POVM and let    $f: X \rightarrow \mathcal{B}(\mathcal H)$ be a quantum random variable. Then the \emph{essential range} of $f$ is the set $\essran f$ of all $A\in \mathcal{B}(\mathcal H)$ for which $\nu(f^{-1}(U))\neq 0$ for every neighbourhood $U\subseteq \mathcal{B}(\mathcal H)$ of $A$. 
\end{definition}

A standard classical result is that the essential range of a measurable function is equal to the intersection of the closure of the image of the function over all measurable sets whose complement has measure zero. We show this is also the case in the setting of quantum random variables. Although the proof is a straightforward generalization of the classical setting (see \cite[Proposiion 5.45]{FarenickBook}), we include it here for completeness and rigor.

\begin{proposition}
Let   $f: X \rightarrow \mathcal{B}(\mathcal H)$ be a quantum random variable. Then 
\[
\essran f=\bigcap_{\substack{E\in \mathcal O(X)\\\nu(E^c)=0}} \overline{f(E)}.
\]
\end{proposition}

\begin{proof}
Let  $A\in \essran f$ and suppose $E\in \mathcal O(X)$ is such that $\nu(E^c)=0$. This implies $A\in \overline{f(E)}$. Indeed, if that were not the case, then there would be an open set $U$ containing $A$ such that $U\cap \overline{f(E)}=\emptyset$, and so $f^{-1}(U)\subset E^c)$; however, $0 = \nu(E^c) \geq \nu(f^{-1}(U))\neq 0$ (since $A\in \essran f$), a contradiction. It follows that $\essran f\subseteq \overline{f(E)}$ for every  $E\in \mathcal O(X)$ is such that $\nu(E^c)=0$.

Conversely,  let $A\notin \essran f$. Then $\nu(f^{-1}(U))=0$ for some neighbourhood $U$ of $A$. Let $E=f^{-1}(U)^c$. This implies $A\notin \overline{f(E)}$. Indeed, if that were not the case, then $f(E)\cap U\neq \emptyset$, but then $E\cap E^c\neq \emptyset$, a contradiction. It follows that $\overline{f(E)}\subseteq \essran f$ for every $E\in \mathcal O(X)$ such that $\nu(E^c)=0$. 
\end{proof} 

Again, analogous to the classical setting, we define the essential supremum of a quantum random variable. 
\begin{definition}
Let   $f: X \rightarrow \mathcal{B}(\mathcal H)$ be a quantum random variable. The \emph{essential supremum} of $f$ is the quantity
\begin{align*}
\esssup f & =\sup\{\|A\|\,:\, A\in \essran f\}.
\end{align*}
If $\esssup f<\infty$ then $f$ is \emph{essentially bounded} and 
\[
\esssup f =\inf\{M \geq 0 : \nu(f^{-1}(\{A\in \mathcal B(\mathcal H) : \|A\| > M\})) = 0\}.
\]
\end{definition}

Let $\nu : \mathcal O(X) \rightarrow \mathcal B(\mathcal H)$ be a POVM such that $\frac{d\nu}{d\nu_\rho}$ exists for any full-rank density operator $\rho$. 
\cite[Corollary 2.9]{MPR} 
tells us that every quantum random variable that is essentially bounded with respect to $\nu$ is $\nu$-integrable. Denote the set of all such quantum random variables by $\mathcal L^\infty_{\mathcal H}(X,\nu)$. 

As in the classical case, for $f\in \mathcal L^\infty_{\mathcal H}(X,\nu)$ let $\|f\|_\infty = \esssup f$ which is a seminorm on $\mathcal L^\infty_{\mathcal H}(X,\nu)$. Quotienting by the ideal of functions that vanish under this seminorm we get $L^\infty_{\mathcal H}(X,\nu)$. Because $\nu$ and $\nu_\rho$ are mutually absolutely continuous they give the same set of essentially bounded quantum random variables from $X$ to $\mathcal B(\mathcal H)$ by the second part of the definition of essentially bounded.
This implies that
\[
L^\infty_{\mathcal H}(X, \nu) = L^\infty_{\mathcal H}(X,\nu_\rho) \simeq L^\infty(X,\nu_\rho) \ \overline\otimes\ \mathcal B(\mathcal H)
\]
is a Banach space, actually a von Neumann algebra (and so the tensor product is unique by \cite[Section IV.5]{Takesaki}). In the same section of the previous reference, there is the following predual identity
\[
(L^\infty_{\mathcal H}(X, \nu))_* \simeq (L^\infty(X,\nu_\rho) \ \overline\otimes\ \mathcal B(\mathcal H))_* = L^1(X,\nu_\rho) \ \overline\otimes \ \mathcal T(\mathcal H).
\]
Thus $L^\infty_{\mathcal H}(X, \nu)$ carries a weak$^*$-topology. To be clear, on sums of simple tensors these isomorphisms are both given by 
\[
\sum_{i=1}^n f_i(x) \otimes A_i \mapsto \sum_{i=1}^n f_i(x)A_i.
\]
So we consider both $L^\infty_\mathcal H(X,\nu)$ and its predual as sets of equivalence classes of functions from $X$ into $\mathcal B(\mathcal H)$ or $\mathcal T(\mathcal H)$, respectively.

Integration is a very important concept in our analysis and so, for a given POVM  $\nu : \mathcal O(X) \rightarrow \mathcal B(\mathcal H)$   such that $\frac{d\nu}{d\nu_\rho}$ exists for any full-rank density operator $\rho$,  we define $\mathcal E_\nu : L^\infty_\mathcal H(X,\nu) \rightarrow \mathcal B(\mathcal H)$ to be
\begin{equation}\label{eq:E}
\mathcal E_\nu(f) = \int_X fd\nu, \ \ f\in L^\infty_\mathcal H(X,\nu).
\end{equation}
This integral operator is called the \emph{quantum expected value} of $f$ relative to the POVM $\nu$ in \cite{FK, JK, FKP} and denoted $\mathbb{E}_\nu(f)$.

\begin{theorem}\label{lem:w*cont}
Let $\nu: \mathcal{O}(X)\rightarrow \mathcal{B}(\mathcal{H)}$ be a POVM such that $\frac{d\nu}{d\nu_\rho}$ exists. Then $\mathcal E_\nu$ is weak$^*$-weak$^*$-continuous. 
\end{theorem}
\begin{proof}
For every $s\in \mathcal T(\mathcal H)$ we first need to establish that 
\[
x \mapsto \left(\frac{d\nu}{d\nu_\rho}(x)\right)^{1/2} s \left(\frac{d\nu}{d\nu_\rho}(x)\right)^{1/2}
\]
is an element of $L^1(X,\nu_\rho) \bar\otimes \mathcal T(\mathcal H)$.
To this end, consider the bounded linear functional $f\mapsto \int_X\tr(f(x)h(x)A)\,d\nu_\rho(x)$, $h\in L^\infty(X,\nu_\rho), A\in \mathcal B(\mathcal H)$ which is precisely integration against $h$ in the first component and taking the trace against $A$ in the second component of $L^1(X, \nu_\rho) \overline\otimes \mathcal T(\mathcal H)$. In particular, the span of these functionals is dense in $(L^1(X, \nu_\rho) \overline\otimes \mathcal T(\mathcal H))^* = L^\infty_\mathcal H(X,\nu)$.

We then have 
\begin{eqnarray*}
&&\int_X\tr\left( \left(\frac{d\nu}{d\nu_\rho}(x)\right)^{1/2} s \left(\frac{d\nu}{d\nu_\rho}(x)\right)^{1/2}h(x)A\right)\,d\nu_\rho(x)\\
&&\quad =\tr\left(s\int_Xh(x)A\,d\nu(x)\right)\in \mathbb{C}
\end{eqnarray*}
by Corollary 2.9 of \cite{MPR} 
since $h(x)A$ is essentially bounded. Therefore, $$\left(\frac{d\nu}{d\nu_\rho}(x)\right)^{1/2} s \left(\frac{d\nu}{d\nu_\rho}(x)\right)^{1/2} \in L^1(X,\nu_\rho)\ \overline{\otimes}\ \mathcal T(\mathcal H).$$

Now suppose $\psi_n \xrightarrow{wk^*} \psi$ in $L^\infty_\mathcal H(X,\nu)$. By the above, every bounded linear functional of $L^1(X, \nu_\rho) \overline\otimes \mathcal T(\mathcal H)$ is given as $f \mapsto \int_X \tr(f\psi)d\nu_\rho$ for $\psi\in L^\infty_\mathcal H(X,\nu)$.
Then
\begin{eqnarray*}
\tr\left(s\int_X\psi_n(x)\, d\nu(x)\right)&=&\int_X\tr\left( \left(\frac{d\nu}{d\nu_\rho}(x)\right)^{1/2} s \left(\frac{d\nu}{d\nu_\rho}(x)\right)^{1/2}\psi_n(x)\right)\,d\nu_\rho(x)\\
&\rightarrow& \int_X\tr\left( \left(\frac{d\nu}{d\nu_\rho}(x)\right)^{1/2} s\left(\frac{d\nu}{d\nu_\rho}(x)\right)^{1/2}\psi(x)\right)\,d\nu_\rho(x)\\
&=&\tr\left(s\int_X\psi(x)\, d\nu(x)\right),
\end{eqnarray*}
from which it follows that $\int_X\psi_n\,d\nu \xrightarrow{wk^*} \int_X\psi\,d\nu$.
\end{proof}


A very natural problem that arises from the study of integration against a POVM is that of whether $\mathcal E_\nu$ is injective or not.
More specifically, for a POVM $\nu: \mathcal O(X) \rightarrow \mathcal B(\mathcal H)$ we say that integration against $\nu$ is \emph{classically non-injective} if for every $E\in \mathcal O(X)$ such that $\nu(E)\neq 0$ we have that 
\[
f\in L^\infty(E,\nu) \ \longmapsto \ \mathcal E_\nu(f I_\mathcal H)
\]
is non-injective.

It is immediate that integration against any measure with an atom fails this condition. By cardinality alone it is also immediate that integration against a non-atomic finite-dimensional POVM is classically non-injective.

Such a condition is quite essential to proving a convexity result as shown by an example of Uhl \cite[Example 1, Chapter IX]{DiestelUhl}.

\begin{example}[Uhl]
Let $\mu$ be Lebesgue measure on $[0,1]$ and $L^\infty([0,1],\mu) \subset \mathcal B(\mathcal H)$. Define the POVM $\nu : \mathcal O([0,1]) \rightarrow \mathcal B(\mathcal H)$ by 
\[
\nu(E) = \chi_E, \quad \forall E\in \mathcal O([0,1]).
\]
Then the range space $\mathcal R_\nu = \{\nu(E) : E\in \mathcal O([0,1])\}$ is not convex. Note that it can be proven that this non-atomic measure is not classicaly non-injective.
\end{example}


\section{Operator-valued Lyapunov theorem}\label{sec:Lyapunov}

We are now in a position to prove the main theorem of this paper. 

\begin{theorem}\label{thm:rangeconvex}
Let $\nu : \mathcal{O}(X) \rightarrow \mathcal{B}(\mathcal{H})$ be a nonatomic POVM such that $\frac{d\nu}{d\nu_\rho}$ exists and let $\mathcal E_\nu: L^\infty_\mathcal H(X,\nu) \rightarrow \mathcal B(\mathcal H)$ be as above. If integration against $\nu$ is classically non-injective  
then the set $\mathcal R_\nu = \{\nu(E) : E\in \mathcal O(X)\}$ is a weak$^*$-compact convex set in $\mathcal B(\mathcal H)$.
\end{theorem}

\begin{proof}
The proof follows the same general structure as in the classical setting (see \cite[Theorem 7.26]{FarenickBook} or \cite{Strauss}).
Let  
 \begin{eqnarray*}
 \mathcal{I}&=&\{\psi\in  L^\infty_{\H}(X,\nu)\,:\, \essran \psi \in \{\lambda I_\mathcal H : \lambda \in [0,1]\}\} 
 \\ &=& \{h\in L^\infty(X,\nu_\rho)\,:\, 0\leq h(x)\leq 1\ \textrm{a.a.}\} \ \overline{\otimes} \ I_\mathcal H.
 \end{eqnarray*}
It is well known from the classical case that the set on the left side of the tensor product is weak$^*$-closed and thus weak$^*$-compact (being in its unit ball). Hence, $\mathcal I$ is weak$^*$-compact as well.
 
 It is easy to see that the map $\mathcal{E}_\nu:\mathcal I\rightarrow \mathcal{B}(\mathcal{H})$ defined in equation (\ref{eq:E}) is affine with respect to quantum random variables $\psi$ (although it is not affine with respect to POVMs $\nu$). Since $\mathcal{E}_\nu$ is weak$^*$-weak$^*$-continuous by Theorem \ref{lem:w*cont} we have that $K=\mathcal{E}(\mathcal I)$ is weak$^*$-compact and convex. 
 
 One should note here that $K \subseteq [0, \nu(X)]$ and contains $0$ and $\nu(X) = \mathcal E_\nu(\chi_X(x)I_\mathcal H)$ but it does not have to equal the operator interval $[0,\nu(X)]$. For instance, $K$ would equal $\{\lambda I_\mathcal H : 0\leq \lambda \leq 1\}$ if $\nu = \mu I_\mathcal H$ where $\mu$ is Lebesgue measure on $[0,1]$.

Pick $A\in K$. The set $\mathcal I_A=\{\psi\in \mathcal I\,:\, \ \mathcal{E}_\nu (\psi)=A\}= \mathcal{E}_\nu^{-1}(\{A\})$ is the weak$^*$-continuous inverse image of a weak$^*$-closed set, and so $\mathcal I_A$ is weak$^*$-closed. Clearly $\mathcal I_A$ is convex as well. By the Krein-Milman theorem, there exists $\psi\in \mbox{ext}\ \mathcal I_A$. We claim that $\psi=\chi_E$ for some $E\in \mathcal{O}(X)$. This will complete the proof since we already have that $ \mathcal{R}_\nu = \{\nu(E) : E\in \mathcal O(X)\} \subseteq K$, giving $\mathcal{R}_\nu=K$ is weak$^*$-compact and convex. 

Now, we prove the claim by contradiction: suppose $\psi\in  \mbox{ext}\ \mathcal I_A$ is \emph{not} a characteristic function; that is, $\psi(x) \neq \chi_E(x)I_\mathcal H$ in $L^\infty_\mathcal H(X,\nu)$. 
Thus, the essential range of $\psi$ contains some $\lambda I_\mathcal H$ with $0< \lambda< 1$. Let $\epsilon = \frac{1}{2}\min\{\lambda, 1-\lambda\}$ and define $U = \{B\in \mathcal B(\mathcal H) : \|B- \lambda I_\mathcal H\| < \epsilon\}$. Note that $U$ is open and $\lambda I_\mathcal H\in U$ but $0,I_\mathcal H \notin U$. We then have that $E = \psi^{-1}(U)$ is a Borel set and $\nu(E) \neq 0$.

 So $E$ is a measurable set with $\nu(E)\neq 0$, because $\lambda I_\mathcal H$ is in the essential range of $\psi$. By the fact that integration against $\nu$ is classically non-injective there exists a $\sigma\neq 0 \in L^\infty(E,\nu)$ such that $\mathcal E_\nu(\sigma I_\mathcal H) = 0$. One can consider $\sigma(x) = 0$ on $X\setminus E$ and 
 hence $\psi\pm \frac{\epsilon}{2\|\sigma\|_\infty}\sigma\in \mathcal I$. Moreover, $\mathcal E_\nu(\psi \pm \frac{\epsilon}{2\|\sigma\|_\infty}\sigma) = A$ and so $\psi \pm \frac{\epsilon}{2\|\sigma\|_\infty}\sigma \in \mathcal I_A$ contradicting the assumption that $\psi\in \mbox{ext}\ \mathcal I_A$, and the claim is proven.
\end{proof}

\begin{corollary}\label{thm:quantumLyapunov}
Let $\nu_i : \mathcal{O}(X) \rightarrow \mathcal{B}(\mathcal{H}_i)$ be nonatomic POVMs such that $\frac{d\nu_i}{(d\nu_{i})_\rho}$ exists for $i=1,2,\dots,n$. If integration against $\nu_i$ is classically non-injective then the set of points in $\bigoplus_{i=1}^n \mathcal{B}(\mathcal{H}_i)$ of the form $(\nu_1(E),\nu_2(E),\dots,\nu_n(E))$, with $E$ ranging over all $\mathcal{O}(X)$, is a weak$^*$-closed convex set.
\end{corollary}
\begin{proof}
Define $\nu : \mathcal O(X) \rightarrow \mathcal B(\mathcal H)$ as $\nu(E) = \bigoplus_{i=1}^n \nu_i(E)$ on $\mathcal H = \bigoplus_{i=1}^n \mathcal H_i$. It is not hard to see that $\nu$ is a nonatomic POVM and $\frac{d\nu}{d\nu_\rho}$ exists. The result follows from the last theorem since integration by $\nu$ is classically non-invertible.
\end{proof}

Note that if $X$ is a discrete set, then the range of $\nu$ cannot be convex, and so the above corollary shows that nonatomic POVMs such that $\frac{d\nu_i}{(d\nu_{i})_\rho}$ exists for $i=1,2,\dots,n$ have fundamentally different range spaces compared to atomic POVMs.

Lyapunov's  Convexity Theorem follows as a one-dimensional consequence of Corollary \ref{thm:quantumLyapunov}:

\begin{corollary}\cite{Strauss}\label{thm:ClassLyapunov} Let $\{\mu_1,\mu_2,\dots,\mu_n\}$  be finite positive nonatomic measures on $\Sigma$. Then the set of points in $\mathbb{R}^n$ of the form $(\mu_1(E),\mu_2(E),\dots,\mu_n(E))$, with $E$ ranging over all $\Sigma$, is a closed convex set.
\end{corollary}

Our methods rely on the existence of the Radon-Nikod\'ym derivative $\frac{d\nu}{d\nu_\rho}$ (which, as we mentioned before, always exists when $\mathcal H$ is finite-dimensional). In the absence of this derivative there are easy examples where one can see that Lyapunov's theorem will still hold but we do not have a proof of the general situation. One such case is the following:

\begin{example}
Let $\{\mu_n\}_{n\geq 1}$ be a set of mutually singular nonatomic probability measures on the measure space $(X, \mathcal O(X))$. Define $\nu(E) = {\rm diag}(\mu_1(E), \mu_2(E), \dots)$ and so $\nu : \mathcal O(X) \rightarrow \mathcal B(\mathcal H)$ is a nonatomic quantum probability measure. By the same reasoning as \cite[Example 2.4]{MPR} it is easy to see that $\frac{d\nu}{d\nu_\rho}$ does not exist (just take $\rho = {\rm diag}(1/2,1/4,\dots)$).

Let $X_n = \esssupp \mu_n$ and by the mutually singular assumption we have that $X_n \cap X_m = \emptyset$. For any infinite tuple $\lambda = (\lambda_1,\lambda_2, \dots)$ with $0\leq \lambda_n\leq 1$ we can find by Lyapunov's theorem above that there exists $E_n\in \mathcal O(X)$ such that $E_n\subseteq X_n$ and $\mu_n(E_n) = \lambda_n$.
Thus, for $E = \cup_{n\geq 1} E_n \in \mathcal O(X)$ we have $\nu(E) = {\rm diag}(\lambda_1,\lambda_2,\dots)$. Therefore, $\mathcal R_\nu = \{\nu(F) : F\in\mathcal O(X)\}$ is a weak$^*$-closed and convex subset in the operator interval $[0,I_\mathcal H]$.
\end{example}

\section*{Acknowledgements}
 S.P.\ was supported by NSERC Discovery Grant number 1174582, the Canada Foundation for Innovation, and the Canada Research Chairs Program. S.P.\ thanks Doug Farenick for helpful discussions at the initial stage of this work.


\end{document}